\documentclass[12pt,reqno]{amsart}
\usepackage{geometry}
\usepackage[latin1]{inputenc}
\usepackage[italian,english]{babel}
\usepackage{amsmath, amsfonts, amsthm, xcolor}
\usepackage{paralist}
\numberwithin{equation}{section}
\usepackage{mathtools, mathabx, verbatim}

\newtheorem{thm}{Theorem}[section]
\newtheorem{lem}[thm]{Lemma}

\newtheorem{prop}[thm]{Proposition}

\theoremstyle{definition}
\newtheorem{rem}[thm]{Remark}
\theoremstyle{remark}

\newcommand{\ds}{\displaystyle}

\newcommand{\R}{\mathbb{R}}
\newcommand{\N}{\mathbb{N}}

\newcommand{\de}{\partial}

\DeclareMathOperator{\spt}{spt}
\DeclareMathOperator{\diam}{diam}
{\left\{\begin{array}{@{}l@{}}}{\end{array}\right.}
\patchcmd{\abstract}{\scshape\abstractname}{\textbf{\abstractname}}{}{}
\makeatletter 
\def\@makefnmark{} 
\makeatother 

\title[Isoperimetric inequality for the first Steklov-Dirichlet Laplacian]{An Isoperimetric inequality for the first  Steklov-Dirichlet Laplacian eigenvalue of convex sets with a spherical hole}
\author[N. Gavitone, G. Paoli, G. Piscitelli,  R. Sannipoli]{
	Nunzia Gavitone, Gloria Paoli, Gianpaolo Piscitelli, Rossano Sannipoli}
\address{Dipartimento di Matematica e Applicazioni ``R. Caccioppoli'', Universit\`a degli studi di Napoli Federico II \\ Via Cintia, Complesso Universitario Monte S. Angelo, 80126 Napoli, Italy.}
\email{nunzia.gavitone@unina.it}
\address{Dipartimento di Matematica e Applicazioni ``R. Caccioppoli'', Universit\`a degli studi di Napoli Federico II \\ Via Cintia, Complesso Universitario Monte S. Angelo, 80126 Napoli, Italy.}
\email{gloria.paoli@unina.it}
\address{Dipartimento di Ingegneria Elettrica e dell'Informazione \lq\lq M. Scarano\rq\rq, Universit\`a degli Studi di Cassino e del Lazio Meridionale\\ Via G. Di Biasio n. 43, 03043 Cassino (FR), Italy.}
\email{gianpaolo.piscitelli@unicas.it}
\address{Dipartimento di Matematica e Applicazioni ``R. Caccioppoli'', Universit\`a degli studi di Napoli Federico II \\ Via Cintia, Complesso Universitario Monte S. Angelo, 80126 Napoli, Italy.}
\email{rossano.sannipoli@unina.it}

\begin{document}
\maketitle
\begin{abstract}
In this paper we  prove the existence of a maximum for the first Steklov-Dirichlet eigenvalue in the class of convex sets with a fixed spherical hole under volume constraint. More precisely,  if  $\Omega=\Omega_0 \setminus \overline{B}_{R_1}$, where     $B_{R_1}$ is the ball centered at the origin with radius $R_1>0$ and  $\Omega_0\subset\mathbb{R}^n$, $n\geq 2$, is an open bounded and convex set such that $B_{R_1}\Subset \Omega_0$, then the first Steklov-Dirichlet  eigenvalue $\sigma_1(\Omega)$ has a maximum when $R_1$ and the  measure of $\Omega$ are fixed. 
Moreover, if $\Omega_0$ is contained in a suitable ball,  we prove that the spherical shell is the maximum.\\ 

\end{abstract}

\section{Introduction and main results}
Let $\Omega_0\subset\mathbb{R}^n$, $n\geq 2$, be an open, bounded, connected set with Lipschitz boundary such that $B_{R_1}\Subset\Omega_0$, where $B_{R_1}$ is the open ball of radius $R_1>0$ centered at the origin  such that its closure is strictly contained in $\Omega_0$ and let us set $\Omega :=\Omega_0\setminus \overline{B_{R_1}}$.
The first Steklov-Dirichlet  eigenvalue  of  $\Omega$ is defined by 
\begin{equation}\label{eigSD_intro}
\sigma_1( \Omega)=\min_{\substack{v\in H^{1}_{\partial B_{R_1}}(\Omega)\\ v\not \equiv0} }     \dfrac{\ds\int _{\Omega}|\nabla v|^2\;dx}{\ds\int_{\partial\Omega_0}v^2\;d\mathcal{H}^{n-1}},
\end{equation}
where $H^{1}_{\partial B_{R_1}}(\Omega)$ is the set of Sobolev functions on $\Omega$ vanishing on $\partial B_{R_1}$ (see Section \ref{prelim} for the precise definition).   

Denoting by  $\nu$  the  outer unit normal to  $\partial\Omega_0$,   any minimizer of \eqref{eigSD_intro} satisfies the following problem  \begin{equation}\label{proSD_intro}
\begin{cases}
\Delta u=0 & \mbox{in}\ \Omega\vspace{0.2cm}\\
\dfrac{\de u}{\de \nu}=\sigma (\Omega)u&\mbox{on}\ \partial\Omega_0\vspace{0.2cm}\\ 
u=0&\mbox{on}\ \partial B_{R_1}, 
\end{cases}
\end{equation}
with   $\sigma(\Omega)=\sigma_1(\Omega)$.  For more details on $\sigma_1(\Omega)$ and the problem \eqref{proSD_intro} we refer the reader to Section $2.2$.

%



When $R_1=0$, \eqref{proSD_intro} is the classical  Steklov-Laplacian  eigenvalue problem. In this case, Weinstock in \cite{W1, W2} proved an isoperimetric inequality for the first non-trivial Steklov eigenvalue in two dimensions. More precisely, he showed that among all simply connected sets of the plane with prescribed perimeter, the disk maximizes the first non-trivial Steklov-Laplacian eigenvalue. In  \cite{BFNT} the authors proved that Weinstock inequality holds true in any dimension, provided they restrict to the class of convex sets with fixed perimeter. In \cite{Br}, it is proved that the ball is still a maximizer for the first non-trivial Steklov eigenvalue among all bounded open sets with Lipschitz boundary of $\mathbb{R}^n$, $n\geq 2$, with fixed volume. Stability and instability results are also studied (for instance we refer to \cite{BDR,BN,GLPT}).

When we consider a spherical hole with homoegeneous  Dirichlet boundary condition, that is $R_1>0$, the Steklov-Dirichlet eigenvalue problem  \eqref{proSD_intro} is substantially  different. 
The study of an eigenvalue problem on sets with a spherical hole is actually a topic of interest (see \cite{BKPS, K1, PPT, PW} and the references therein). In particular, problem \eqref{proSD_intro}  has been recently considered by several authors 
(see for instance \cite{D, Ft, HLS, HP,  GP, PPS, VS}). In \cite{PPS} it is proved that the first eigenvalue $\sigma_1(\Omega)$, as defined in \eqref{eigSD_intro}, is bounded from above when both the volume of $\Omega$ and the radius of the inner ball are fixed among the class of nearly spherical sets of $\mathbb{R}^n$. Moreover the authors prove that the spherical shell is a local maximizer.

The aim of this paper is twofold. First, we  prove the existence of a maximum for $\sigma_1(\Omega)$ in the class of sets $\Omega=\Omega_0 \setminus \overline{B}_{R_1}$, where $\Omega_0\subset\mathbb{R}^n$, $n\geq 2$, is an open bounded and convex set containing $B_{R_1}$, keeping  $R_1$ and the  measure of $\Omega$  fixed.  Actually, we prove  this existence result also when the hole is not spherical, but it is an open,  convex set $K\Subset\Omega_0$ with non-empty interior.  

Our second   aim is to find the shape of the maximum when the hole is spherical.
  We stress that, when $\Omega_0=B_{R_2}(x_0)$ is  a ball  centered at $x_0$ with radius $R_2>R_1$, in  \cite{Ft,VS} it is proved that $\sigma_1(\Omega)$ achieves the maximum  when $\Omega$ is the spherical shell, that is when the two balls are concentric.

Our   goal is to prove that this is also true for a suitable class of annular sets. 
More precisely our main result is the following.
\begin{thm}\label{main}
Let $R_1>0$, $\Omega_0\subset \mathbb R^n$ be an open, bounded and convex set, $n \ge 2$, such that $B_{R_1} \Subset  \Omega_0 \subseteq B_{\bar{R}}$, where $B_{\bar R}$ is the ball centered at the origin with radius $\bar R$ given by 
\begin{equation}\label{bound}
\bar R=\begin{cases}
R_1 e^{\sqrt 2} & \text{ if } n=2\\
 R_1 \bigg[\frac{(n-1)+ (n-2)\sqrt{2(n-1)}}{n-1}\bigg]^{\frac{1}{n-2}}  & \text{ if }  n\ge3.
\end{cases}
\end{equation}
Then,  denoting  by $\Omega=\Omega_0 \setminus\overline B_{R_1}$, the following inequality holds
\begin{equation}\label{iso_into}
\sigma_1(\Omega)\leq\sigma_1(A_{R_1,R_2}),
\end{equation}
where $A_{R_1,R_2}$ is the spherical shell of radii $R_1<R_2$ having the same volume as $\Omega$.
\end{thm}
\bigskip

We observe that  the convexity assumption is not just technical but it is  natural  when dealing with Steklov-Dirichlet  eigenvalues (see \cite{FS}). 



The outline of the paper is the  following. In the next  Section we set  the notation
and collect  some basic results about the Steklov-Dirichlet eigenvalue problem that will
be needed in the sequel. In Section $3$, firstly we prove some estimates for  $\sigma_1(\Omega)$ in terms of suitable geometrical quantities related to $\Omega$ and then we state and prove the existence result.  In Section $4$, by
using a suitable weighted  isoperimetric inequality, we prove that the spherical shell is a maximum for $\sigma_1(\Omega)$ when both the volume of $\Omega$ and the radius of the inner ball are fixed. Eventually, in Section $5$, we discuss on a maximization problem for $\sigma_1(\Omega)$ under  perimeter constraint.

\section{Preliminary results}\label{prelim}

\subsection{Notations and basic facts}
Throughout this paper, we denote by $B_{R}(x_0)$ the ball centered at  $x_0\in\mathbb{R}^n$ with radius $R>0$, by $B_R$ the ball centered at the origin with radius $R$ and by $B$, $\mathbb{S}^{n-1}$ and $\omega_n$ respectively the unit ball of $\mathbb R^{n}$ centered at the origin, its boundary and its volume. Let $R_{1},R_{2}$ be such that $0<R_{1}<R_{2}$, the spherical shell  will be denoted as follows:
\begin{equation*}
A_{R_{1},R_{2}}=\{x \in \mathbb R^{n} \colon R_1<|x|<R_2\}.
\end{equation*}
Moreover, the $(n-1)$-dimensional Hausdorff measure in $\mathbb{R}^n$ will be denoted by $\mathcal H^{n-1}$.  The Euclidean scalar product in $\mathbb{R}^n$ is denoted by $(\cdot,\cdot)$.

Let $D\subseteq\mathbb{R}^n$ be an open bounded set and let $E\subseteq\R^{n}$ be a measurable set.  For the sake of completeness, we recall here the definition of the perimeter of $E$ in $D$ (see for instance \cite{Am, M}), that is
\begin{equation*}
P(E;D)=\sup\left\{  \int_E {\rm div} \varphi\:dx :\;\varphi\in C^{\infty}_c(D;\mathbb{R}^n),\;||\varphi||_{\infty}\leq 1 \right\}.
\end{equation*}
The perimeter of $E$ in $\mathbb{R}^n$ will be denoted by $P(E)$ and, if $P(E)<\infty$, we say that $E$ is a set of finite perimeter. Moreover, if $E$ has Lipschitz boundary, it holds
\[
P(E)=\mathcal H^{n-1}(\partial E).
\]
The Lebesgue measure of a measurable set $E \subset \mathbb R^{n}$ will be denoted by $V(E)$.
Moreover, we define the inradius of $E\subset \mathbb R^{n}$ as 
\begin{equation}\label{inradius_def}
  \rho( E)=\sup_{x\in E}\; \inf_{y\in\partial E} |x-y|,
\end{equation} 
while the diameter of $E$ is
\[
\diam (E)=\sup_{x,y \in E} |x-y|.
\]
If $E$ is an open, bounded and convex set of $\R^n$ with non-empty interior, we have (see for instance \cite{DPBG,Sc})
\begin{equation}\label{Cheeger}
	\rho(E)\leq \dfrac{nV(E)}{P(E)},
\end{equation}
and the following  (see \cite{EFT,GWW,Sc}):
\begin{equation}\label{vol_diam}
    P(E)^{n-1}>\omega_{n-1} n^{n-2} \diam(E) V(E)^{n-2}.
\end{equation}
Finally, we recall the definition of Hausdorff distance between two non-empty compact sets $E,F \subset \R^n$, that is (see for instance \cite{Sc}) 
\[
\delta_{\mathcal H}(E,F)=\inf \{\varepsilon >0\ :\ E\subset F+B_\varepsilon ,\ F\subset E + B_\varepsilon \}.
\]
Note that, if $E,F$ are both convex sets, then $\delta_{\mathcal H}(E,F)=\delta_{\mathcal H}(\partial E,\partial F)$. 	
	
Let $\{E_k\}_{k\in\N}$ be a sequence of non-empty compact subsets of $\R^n$, we say that $E_k$  converges to $E$ in the Hausdorff sense and we denote
\[
E_k\stackrel{\mathcal H}{\longrightarrow} E
\]
if and only if  $\delta_{\mathcal H}(E_k,E)\to 0$ as $k\to \infty$.
Moreover,  we say that $\{E_k\}_{k\in\N}$ converges in measure to $E$, and we write $E_k\rightarrow E$, if $\chi_{E_k}\rightarrow\chi_E$ in $L^1(\mathbb R^{n})$, where $\chi_E$ and $\chi_{E_k}$ are the characteristic functions of $E$ and $E_k$ respectively.
In what follows we recall some properties of the convex bodies, i.e. compact convex sets without empty interior. We refer to  \cite{Sc} for further properties and the details.

 Let $K \subset \R^n$ be a bounded convex body. The support function $h_K$ of $K$ is defined as
follows
\[
h_K \colon \mathbb S^{n-1}\to  \R, \quad  h_K(x) = \sup_{y \in K}(x,y).\]
If the origin belongs to $K$ then $h_K$ is non-negative and $h_K(x) \leq \text{diam}(K)$ for every $ x \in \mathbb S^{n-1}$. 	
 Let $K \subset \R^n$ be a bounded convex body such that the origin is an interior point of $K$. The radial function of $K$ is defined as follows
 \begin{equation} 
 \label{ro}
 \rho_K(x)=\sup\{\lambda\ge 0 \colon \lambda x \in K\}, \quad x\in \mathbb S^{n-1},
 \end{equation}
and it is a Lipschitz function. The  radial map  is 
\begin{equation}
\label{r}
r_K\colon \mathbb S^{n-1} \to \partial K, \quad r_K(x)=x\rho_K (x) .
\end{equation}
Then    
\begin{equation}
\label{rad}
\partial K = \{ x\,\rho_K(x),  x \in \mathbb S^{n-1}\}.
\end{equation}
Let us define the minimum and the maximum distance of $\partial K$ from the origin as follows
\begin{equation}
\label{re}
R_m=\min_{\mathbb S^{n-1}}\rho_K(x), \qquad\quad R_M=\max_{\mathbb S^{n-1}}\rho_K(x).
\end{equation}
Moreover if $f\colon \partial K \to \R$ is $\mathcal{H}^{n-1}-$integrable  the following formula for the change of variable given by the radial map holds:
\begin{equation}
\label{cambio}
\int_{\partial K} f \, d \mathcal H^{n-1}= \int_{\mathbb S^{n-1}}f(r_K(x)) \displaystyle \frac{\rho_K(x)}{h_K(\nu_K(r_K(x)))}\, d \mathcal H^{n-1},
\end{equation}
where $\nu_K(r_K(x))$ is the outer unit normal to $\partial K$ at the point $r_K(x)=x\rho_K(x)$. We have  (see for example \cite{Sc}) 
\[
\nu_K(r_K(x))=\displaystyle \frac{x\rho_K(x)-\nabla_{\tau}\rho_K(x)}{\sqrt{(\rho_K(x))^2+|\nabla_{\tau}\rho_K(x)|^2}},
\]
where by  $\nabla_{\tau} \rho_K $ we denote the 
 the component of  $\nabla \rho_K$ tangential to $\mathbb S^{n-1}$.
So, we observe that \eqref{cambio} is equivalent to 
\[
\int_{\partial K} f \, d \mathcal H^{n-1}=\int_{\mathbb S^{n-1}}f(r_K(x))(\rho_K(x))^{n-1}\sqrt{1+\left(\frac{|\nabla_\tau \rho_K(x)|}{\rho_K(x)}\right)^2}\,d\mathcal H^{n-1}.
\]
The following result holds (see for instance \cite{chen}, \cite{wang}, \cite{Sc}).
\begin{lem}
\label{conv}
Let $K_n$ and $K$ be bounded convex bodies containing the origin for any $n \in \mathbb N$ and such that $K_n \to K$  in the Hausdorff sense. For any  $n\in\{0,1,2,...\}$, let $h_{K_n}, \rho_{K_n}$ be the support function and the radial function  $K_n$, respectively. Then the following statements hold
\begin{itemize}
\item[(i)] Let $h_K$ be the support function of $K$ then
\[
\sup_{x\in\mathbb S^{n-1}}|h_{K_n}(x)-h_K(x)| \to 0.
\]
\item[(ii)] Let $\rho_K$ the radial function of $K$ then
\[
\sup_{x\in\mathbb S^{n-1}}|\rho_{K_n}(x)-\rho_K(x)| \to 0.
\]
\item[(iii)]  Let $x\in \partial K$ and $x_n\in \partial K_n$, $n\in\mathbb{N}$, points where $\nu_K(x)$ and $\nu_{K_n}(x_n)$ are well defined and such that
\begin{equation*}
\lim_{n\to+\infty} x_n=x.
\end{equation*}
Then 
\begin{equation*}
\lim_{n\to +\infty}\nu_{K_n}(x_n)=\nu_K(x).
\end{equation*}
\end{itemize}
\end{lem}
By \eqref{cambio}, Lemma \ref{conv} and the Lebesgue's convergence Theorem  we immediately get
\begin{thm}
\label{fr}
Let $K_n$ and $K$ be bounded convex bodies containing the origin for any $n \in \mathbb N$ and such that $K_n \to K$  in the Hausdorff sense. Let $$f_n \colon \partial K_n \to \R, \quad  f \colon 	\partial K \to \R$$ be $\mathcal H^{n-1}$ measurable functions such that 
\begin{itemize}
\item [(i)] there exists $C>0$ such that\[\|f\|_{L^{\infty}(\partial K)}\le C, \quad \|f_n\|_{L^{\infty}(\partial K_n)}\le C, \,\, \forall n \in +\mathbb N,\] 
\item[(ii)] if $x_n\in \partial K_n$ is such that $ \lim_{n\to\infty} x_n=x \in \partial K$, $f_n$ is defined in $x_n$ and
\[
\lim_{n \to+ \infty} f_n(x_n)= f(x).
\]
\end{itemize}
Then
\[
\lim_{n \to +\infty} \displaystyle \int_{\partial K_n} f_n(x_n) \,d \mathcal H^{n-1}=\displaystyle \int_{\partial K} f(x) \,d \mathcal H^{n-1}.
\]
\end{thm}

\subsection{The Steklov-Dirichlet eigenvalue problem}
Let $R_1>0$ and $\Omega_0\subset\R^n$ be an open bounded connected set with Lipschitz boundary and such that  $B_{R_1}\Subset\Omega_0$, that means $\overline B_{R_1}\subset\Omega_0$. 
Let us  set $\Omega :=\Omega_0\setminus \overline B_{R_1}$. 

We denote the set of Sobolev functions on $\Omega$ vanishing on $\partial B_{R_1}$ by $H^1_{\partial B_{R_1}}(\Omega)$, that is (see \cite{ET}) the closure in $H^1(\Omega)$ of the following set

\begin{equation*}
C^\infty_{\partial B_{R_1}} (\Omega):=\{ u_{|\Omega}  \ | \ u \in C_0^\infty (\R^n),\ \spt (u)\cap \partial B_{R_1}=\emptyset \}.  
\end{equation*}
Let us consider the following Steklov-Dirichlet eigenvalue problem in $\Omega$ 
\begin{equation}\label{problem}
\begin{cases}
\Delta u=0 & \mbox{in}\ \Omega\vspace{0.2cm}\\
\dfrac{\de u}{\de \nu}=\sigma (\Omega)u&\mbox{on}\ \partial\Omega_0\vspace{0.2cm}\\ 
u=0&\mbox{on}\ \partial B_{R_1}, 
\end{cases}
\end{equation}
where  $\nu $ is the  outer unit normal to  $\partial\Omega_0$.
The spectrum of \eqref{problem} is  discrete and the eigenvalue can be ordered as follows
\[
0<\sigma_{1}(\Omega)\leq\sigma_2(\Omega)\leq...\ .
\]
In  \cite{PPS} the authors study  the first eigenvalue of \eqref{problem},  which has the following variational characterization
\begin{equation}\label{eig}
\sigma_1( \Omega)=\min_{\substack{v\in H^{1}_{\partial B_{R_1}}(\Omega)\\ v\not \equiv0} }     
\dfrac{\ds\int _{\Omega}|\nabla v|^2\;dx}{\ds\int_{\partial\Omega_0}v^2\;d\mathcal{H}^{n-1}}, 
\end{equation}
and they also prove the following result.
\begin{prop}
Let $R_1>0$ and $\Omega_0\subset\R^n$ be an open bounded connected set with Lipschitz boundary and such that  $B_{R_1}\Subset\Omega_0$ and let $\Omega :=\Omega_0\setminus \overline B_{R_1}$. There exists a function $u\in H^1_{\partial B_{R_1}}(\Omega)$ which achieves the minimum  in \eqref{eig} and it is a weak solution to the problem \eqref{problem}. Moreover $\sigma_1(\Omega)$ is simple and the first eigenfunctions have constant sign in $\Omega$. 
\end{prop}

In particular they prove the following upper bound
\begin{equation}\label{upper_bound}
\sigma_1(\Omega)\leq   C(n,R_1,V(\Omega))V^{\frac1n}(\Omega),
\end{equation}
where 
\begin{equation*}
C(n, R_1, V(\Omega))=\dfrac{2}{n\omega_n^{\frac 1n}\left(\left(   \dfrac{V(\Omega)}{2\omega_n}+R_1^n\right)^{\frac1n}-R_1\right)^2}.
\end{equation*}
As a consequence, the first Steklov-Dirichlet eigenvalue remains bounded from above, when  the volume of $\Omega$ and $R_1$ are fixed.

Obviously $\sigma_1(\Omega)$ is  bounded also when we fix the perimeter of $\Omega$, that is equivalent to fix the perimeter of $\Omega_0$, instead of the volume. Indeed by \eqref{upper_bound} and the isoperimetric inequality, we can deduce the following upper bound
\begin{equation}
\label{Upperp}
\sigma_1(\Omega)\leq \dfrac{2V^{\frac1n}(\Omega)}{n\omega_n^{\frac 1n}\left(\left(   \dfrac{V(\Omega)}{2\omega_n}+R_1^n\right)^{\frac1n}-R_1\right)^2}\leq C(n)\dfrac{P^{\frac{1}{n-1}}(\Omega_0)}{R_1^2},
\end{equation}
where $C(n)$ is a positive constants that depends only on the dimension $n$.


In particular the following scaling property for $\sigma_1(\Omega)$ holds: 
\begin{equation*}
    \sigma_1(t\Omega)=\frac 1 t \sigma_1(\Omega), \quad\forall t>0.
\end{equation*}

Now we recall some known results about the first Steklov-Dirichlet eigenvalue when $\Omega$ is  a spherical shell. In this case,  in  \cite{VS} the authors find the explicit expression of the first eigenfunction.
\begin{prop} Let  $A_{R_1,R_2}$ be the spherical shell with radii $R_2> R_1>0$. The first eigenfunction associated to $\sigma_1(A_{R_1,R_2})$, is a  radially symmetric, positive, strictly increasing function and it is given by	
\begin{equation}
\label{radial}
w(r)=\begin{cases}
\ln r-\ln R_1
& {\rm for}\;\; n=2\vspace{0.1cm}\\
\left( \dfrac{1}{R_1^{n-2}}-\dfrac{1}{r^{n-2}}\right)& {\rm for}\;\; n\geq 3\vspace{0.1cm},
\end{cases}
\end{equation}
with $r=|x|$. The corresponding first Steklov-Dirichlet eigenvalue can be computed and it is the following
\begin{equation}\label{eig_an}
\sigma_1(A_{R_1,R_2})=
\begin{cases}
\frac{1}{R_2\log\left(\frac{R_2}{R_1}\right)}& {\rm for}\;\; n=2\vspace{0.1cm}\\
\frac{n-2}{R_2\left[\left(\frac{R_2}{R_1}\right)^{n-2}-1\right]}& {\rm for}\;\; n\geq 3\vspace{0.1cm}.\\ 
\end{cases}
\end{equation}
\end{prop}
\begin{rem}
We point out that by \eqref{eig_an},  we have that $\sigma_1(A_{R_1,R_2})$ is increasing with respect to the  radius  of the inner ball, $R_{1}$,  that is\begin{equation*}
\sigma_1(A_{R_1,R_2})<\sigma_1(A_{r_1,R_2}), \quad \text{if } r_{1}>R_{1}.
\end{equation*}
 \end{rem}
Moreover it holds
\begin{equation}\label{lim_0}
\lim_{R_{1}\to 0}\sigma_1(A_{R_1,R_2})=0,
\end{equation}
that is $\sigma_1(A_{R_1,R_2})$ tends to  the first trivial Steklov eigenvalue of the Laplacian for $R_1$ which goes to zero. Finally we stress that an easy computation gives that  $\sigma_1(A_{R_1,R_2})$ is decreasing with respect to the external radius $R_{2}$, that is 
\begin{equation*}\label{mon}
\sigma_1(A_{R_1,R_2})<\sigma_1(A_{r_1,\bar{R}}), \quad \text{if } \bar R<R_{2}.
\end{equation*}

\section{Upper and lower bounds for $\sigma_1(\Omega)$ and existence result}

In this Section  we prove an upper and lower  bound for $\sigma_1(\Omega)$   in terms of $R_m$ and  $R_M$, that are the minimal and  maximal distance from the origin of the outer boundary as defined in \eqref{re}.
Then, we prove  an existence results for a maximizer among convex sets with fixed inner ball and fixed volume and we also generalize it in the case of a suitable  not spherical hole.

\subsection{ Estimates  in terms of $R_m$ and $R_M$}
The proof follows  an idea used in \cite{KS} for the planar case  and in \cite{GM,V} for any dimension to obtain a lower bounds for the first Steklov Laplacian eigenvalue.
\begin{thm}
\label{lower}
Let $R_1>0$ and $\Omega_0\subset\R^n$ be an open bounded connected set with Lipschitz boundary such that  $B_{R_1}\Subset\Omega_0$ and  let $\Omega=\Omega_0\setminus \overline B_{R_1}$.  Then, it holds
\begin{equation}
\label{lb}
\ds \frac{1}{ \max_{\mathbb S^{n-1}} \left(\sqrt{1+ \frac{|\nabla_\tau \rho_0|^2}{\rho_0^2}}\right)} \left(\frac{R_m}{R_M}\right)^{n-1}  \ds \sigma_1 \left(A_{R_1,R_m}\right) \le\sigma_1(\Omega) \le \left(\frac{R_M}{R_m}\right)^{n-1}\sigma_1(A_{R_1,R_M}),
\end{equation}
where $R_m$ and $R_M$ are defined in \eqref{re}, $\rho_0$ is the radial function of $\Omega_0$ defined in \eqref
{ro} and $A_{R_1,R_m}$ is the spherical shell with radii $R_1$ and $R_m$.\\
Moreover, the equality case holds if and only if $\Omega$ is a ball $B_R$ centered at the origin of radius $R>0$.

\end{thm} 
\begin{proof}
Let $u\in H^1_{\partial B_{R_1}}(\Omega)$  be a positive eigenfunction for $\sigma_1(\Omega)$, then 
\begin{equation}
\label{q}
\sigma_1(\Omega)=\dfrac{\ds\int _{\Omega}|\nabla u|^2\;dx}{\ds\int_{\partial\Omega_0}u^2\;d\mathcal{H}^{n-1}}.
\end{equation}
By using spherical  coordinates and the notation introduced in Section $2$:
\[
\partial \Omega_0 = \{x\,\rho_0(x),  x \in \mathbb S^{n-1}\},
\]
the denominator in \eqref{q} becomes
\begin{equation*}
\ds\int_{\partial\Omega_0}u^2\;d\mathcal{H}^{n-1}= \ds \int_{\mathbb S^{n-1}} u^2 \,\,\sqrt{1+\left(\frac{|\nabla_\tau \rho_0|}{\rho_0}\right)^2}\,(\rho_0)^{n-1}\,d\mathcal H^{n-1}.
\end{equation*}
Then, we have
\begin{equation}
\label{den}
(R_m)^{n-1}  \int_{\mathbb S^{n-1}} u^2 \,\,d\mathcal H^{n-1} \le \ds\int_{\partial\Omega_0}u^2\;d\mathcal{H}^{n-1} \le (R_M)^{n-1} \max_{\mathbb S^{n-1}} \left(\sqrt{1+ \frac{|\nabla_\tau \rho_0|^2}{\rho_0^2}}\right) \int_{\mathbb S^{n-1}} u^2 \,\,d\mathcal H^{n-1}.
\end{equation}
Let us now take into account the numerator in \eqref{q}. Since
\[
\Omega=\{s\in \mathbb R^n \colon s=x\,r ,\, x \in \mathbb S^{n-1},\, R_1\le r\le \rho_0(x)  \} ,
\]
by using spherical coordinates and denoted by $R(y)=\rho_0(x(y))$, where  $x\colon y\in U\subset \mathbb R^{n-1}\to x(y) \in \mathbb S^{n-1}$ is a standard parametrization of the boundary of the unit ball in $\mathbb R^{n}$,   we get: 
\begin{equation}
\label{den1}
\ds\int _{\Omega}|\nabla u|^2\;ds=\displaystyle \int_{U} \int_{R_1}^{R(y)}\left\{\left(\frac{\partial u}{\partial r}\right)^2+\frac{1}{r^2}|\nabla_{\tau} u|^2\right\}r^{n-1}\sqrt{\tilde g}\, dr\,dy,\\
\end{equation}
where  $\sqrt {\tilde g}$ is the determinant of the matrix $\tilde g_{ij}$, that is the standard  metric on $\mathbb S^{n-1}$ and $\nabla_{\tau} u$ is the component of $\nabla u$ tangential to $\mathbb S^{n-1}$. Then, we have
\begin{multline}
\label{num}
\displaystyle \int_{U} \int_{R_1}^{R_m}\left\{\left(\frac{\partial u}{\partial r}\right)^2+\frac{1}{r^2}|\nabla_{\tau} u|^2\right\}r^{n-1}\sqrt{\tilde g}\, dr\,dy,\le\ds\int _{\Omega}|\nabla u|^2\;ds\le\\ \le \displaystyle \int_{U} \int_{R_1}^{R_M}\left\{\left(\frac{\partial u}{\partial r}\right)^2+\frac{1}{r^2}|\nabla_{\tau} u|^2\right\}r^{n-1}\sqrt{\tilde g}\, dr\,dy.
\end{multline}
Combining \eqref{den} and \eqref{num} and recalling \eqref{q},  we get
\begin{multline}
\ds \frac{1}{ \max_{\mathbb S^{n-1}} \left(\sqrt{1+ \frac{|\nabla_\tau \rho_0|^2}{\rho_0^2}}\right) } \left(\frac{R_m}{R_M}\right)^{n-1}  \ds \sigma_1 \left(A_{R_1,R_m}\right)  = \\\ds \frac{\displaystyle \int_{U} \int_{R_1}^{R_m}\left\{\left(\frac{\partial u}{\partial r}\right)^2+\frac{1}{r^2}|\nabla_{\tau} u|^2\right\}r^{n-1}\sqrt{\tilde g}\, dr\,dy,
}{\ds(R_M)^{n-1} \max_{\mathbb S^{n-1}} \left(\sqrt{1+ \frac{|\nabla_\tau \rho_0|^2}{\rho_0^2}}\right) \ds \int_{\mathbb S^{n-1}} u^2 \,\,d\mathcal H^{n-1}}
\le \sigma_1(\Omega) \le \\ \le\ds \frac{\displaystyle \int_{U} \int_{R_1}^{R_M}\left\{\left(\frac{\partial u}{\partial r}\right)^2+\frac{1}{r^2}|\nabla_{\tau} u|^2\right\}r^{n-1}\sqrt{\tilde g}\, dr\,dy,
}{\ds(R_m)^{n-1}  \ds\int_{\mathbb S^{n-1}} u^2 \,\,d\mathcal H^{n-1}} =\left(\frac{R_M}{R_m}\right)^{n-1}\sigma_1(A_{R_1,R_M}).
\end{multline}   

 Finally, we stress that  the equality case implies that  all the inequalities become equalities. So, we have that  $\nabla_{\tau}\rho_0=0$ and  $\rho_0(x)=R$, with $R>R_1$ constant.
 \end{proof}

\begin{rem}
We observe that the  lower bound in \eqref{lb} gives that $\sigma_1(\Omega)>0$ being  $R_1>0$ fixed. Moreover, \eqref{lb} also implies a continuity results: $\sigma_1(\Omega) \to 0 $ as  $R_1 \to 0$.
\end{rem}

\subsection{An upper bound for $\sigma_1(\Omega)$ for not spherical hole}

In this subsection we prove an upper bound in the case of a not spherical hole. \\
Let  $K\subset\mathbb{R}^n$ be  a convex body such that $K\Subset\Omega_0$ and let  $\Omega_K=\Omega_0\setminus \overline{K}$. In this case, according to \cite{ET}, the natural space of functions that we have to consider to define $\sigma_1(\Omega_K)$ are $C^\infty_{\partial K}(\Omega_K)$ and $H^1_{\de K}(\Omega_K)$. In particular the classical arguments of Calculus of Variations apply, as showed in \cite{PPS}, and $\sigma_1(\Omega_K)$ is well defined. 
Let us now assume that the volume $|\Omega|=\omega$ and the inradius  $\rho (K)=\tilde{r}$ of $K$  are fixed.  
		Let us consider $A_{\tilde{r},\tilde{R}}$ the spherical shell with radii $\tilde{r}$ and $\tilde{R}$, where $\tilde{R}$ is such that $|A_{\tilde{r},\tilde{R}}|=|\Omega|/2$. So, we have
		\begin{equation}\label{expl_radius}
		\tilde{R}=\left( \frac{|\Omega|}{2\omega_n}+\tilde{r}^n  \right)^{1/n}.
		\end{equation}
		We also consider the following test function $\varphi: \mathbb{R}^n\setminus K\to [0,\infty)$:
		\begin{equation}\label{test_gen}
		\varphi(x)=\begin{cases}
		d_K(x)\quad &\text{ if }  0 \leq d_K(x)\leq \tilde{R}\\
		\tilde{R}\quad  &\text{ if } d_K(x)\geq \tilde{R}
		\end{cases}, 
		\end{equation}
		where 
		$$ d_K(x):=\inf_{y\in\partial K}||x-y||.$$
		and we denote by $K_t$ the set
		\begin{equation}
		K_t=\{ x\in\mathbb{R}^n\setminus \overline{K}\;|\; d_K(x)<t \}.
		\end{equation}
		We have now to distinguish two cases.
		If $K_{\tilde{R}}\Subset\Omega$, then, using the test function \eqref{test_gen} in the variational characterization,  we have
		\begin{equation*}
		\begin{split}
		\sigma_1(\Omega)\leq \dfrac{\ds\int_{K_{\tilde{R}}} |\nabla d_K|^2\;dx}{\ds\int_{\partial\Omega_0}  d^2_K\;d\mathcal{H}^{n-1}} =\dfrac{|K_{\tilde{R}}|}{\tilde{R}^2 P(\Omega_0)}\leq \dfrac{|\Omega|}{\tilde{R}^2 n \omega_n^{1/n}|\Omega|^{1-1/n}}\\=\dfrac{|\Omega|^{1/n}}{n\omega_n^{1/n}\left(\frac{|\Omega|}{2\omega_n}+\tilde{r}^n\right)^{2/n}}=C(n,\tilde{r},|\Omega|),
		\end{split}
		\end{equation*}
		where we have used the  fact that $|\nabla d_K(x)|=1$ a.e., the classical isoperimetric inequality and \eqref{expl_radius}. \\
		Finally, let us consider the case when $K_{\tilde{R}}\not\subseteq\Omega$. We will use the following notations: $\partial ^i\Omega_0=\partial\Omega_0 \cap K_{\tilde{R}}$ and $\partial^e\Omega_0=\partial \Omega_0\setminus\partial^i\Omega_0$. Using as before the  test function \eqref{test_gen}, we have 
		\begin{equation} \label{UCH}
		\begin{split}
		\sigma_1(\Omega)\leq
		\dfrac{\ds\int_{K_{\tilde{R}}\cap \Omega}|\nabla d_K|^2 dx}
		{\ds\int _{\partial^i\Omega_0}  d^2_K d\mathcal{H}^{n-1}   +\int_{\partial^e\Omega_0} \tilde{R}^2 d\mathcal{H}^{n-1} }
		\leq \dfrac{|K_{\tilde{R}}\cap \Omega|}{\tilde{R}^2 |\partial^e\Omega_0|}\\\leq \dfrac{2 |\Omega|}{\tilde{R}^2n\omega_n^{1/n}|\Omega|^{1-1/n}}=2C(n,\tilde{r},|\Omega|),
		\end{split}
		\end{equation}
		where we have used the relative isoperimetric inequality (see \cite[Proposition 2.4]{PPS} and the references therein).\\

\subsection{The existence result}

Inequality \eqref{upper_bound} ensures that the Steklov-Dirichlet eigenvalue $\sigma_1(\Omega)$, defined in \eqref{eigSD_intro}, is bounded from above if the volume of $\Omega$ is fixed. In this section we prove the existence of a maximizer among convex sets with  fixed internal ball and fixed volume.
Let  $\omega>0$ and $R_1>0$ be fixed, then by $\mathcal{A}_{R_1}(\omega)$ we will denote  the class of convex sets having measure $\omega$ and containing the ball $B_{R_1}$, that is
\begin{equation*}
\mathcal{A}_{R_1}(\omega) := \big\{D = K \setminus \overline{B}_{R_1},\:\ K\subseteq\mathbb{R}^n \,\,\text{open, bounded, convex}:\ B_{R_1}\Subset K,\  V(D)= \omega \big \}.
\end{equation*}
The main theorem of this section is the following existence result.
\begin{thm}\label{esi}
	Let $\omega>0$ and $R_1>0$ be fixed. There exists a set $E\in\mathcal{A}_{R_1}(\omega) $, such that 
	\begin{equation*}
	\max_{D\in\mathcal{A}_{R_1}(\omega)}  \sigma_1(D)=\sigma_1(E).
	\end{equation*}
\end{thm}

\begin{proof} 
	
The upper bound \eqref{upper_bound} implies that there exists $M>0$ such that     $$\sup_{D \in \mathcal{A}_{R_1}(\omega)} \sigma_1(D) = M < +\infty.$$ Hence, there exists a sequence $\{E_k\}_{k\in\N} \subseteq \mathcal{A}_{R_1}(\omega) $ such that 
\begin{equation*}
\lim_{k\to\infty} \sigma_1(E_k) = M. 
\end{equation*}
In order to show the desired result,  we need to prove the existence of a set $E\in\mathcal{A}_{R_1}(\omega)$ such that  $E_k\stackrel{\mathcal H}{\longrightarrow} E$ with $\sigma_1(E)=M$.

Firstly we prove that, up to a subsequence,  
 $\{E_k\}_{k\in\N}$ converges to a certain $E\in\mathcal{A}_{R_1}(\omega)$ in the Hausdorff metric. 
 
Being $\{E_k\}_{k\in\N} \subseteq \mathcal{A}_{R_1}(\omega) $
then, for every $k \in \mathbb N$ there exists a convex set $E_{0,k}$, such that $B_{R_1} \Subset E_{0,k}$,
\begin{equation*}
	E_k = E_{0,k} \setminus \overline B_{R_1}
\end{equation*}
and
$$\omega_0:=V(E_{0,k}) = \omega + \omega_n R_1^n.$$
By the Blaschke selection Theorem and the continuity of the volume functional with respect to the Hausdorff measure (see \cite{Sc} as a reference),  it is enough to show that $\{E_{0,k}\}_{k\in\N}$ is equibounded.

 We proceed by contradiction  assuming that 
\begin{equation}\label{diam}
\lim_{k\to+\infty} {\rm diam}(E_{0,k})=+\infty.
\end{equation}
Inequality \eqref{Cheeger} gives
\begin{equation}\label{Cheeger_type}
	\rho(E_{0,k})\leq \dfrac{nV(E_{0,k})}{P(E_{0,k})},
\end{equation}
where $\rho(E_{0,k})$ is the inradius of  $E_{0,k}$ defined in \eqref{inradius_def}. 

The assumption \eqref{diam} and the inequality \eqref{vol_diam} imply  that the right-hand side in \eqref{Cheeger_type} tends to $0$ as $k\to+\infty$, being $V(E_{0,k})$ fixed. Therefore, by \eqref{Cheeger_type},  we have
$$ \lim_{k\to+\infty}\rho(E_{0,k})=0,$$
which is in contradiction with 
$$0<R_{1}<\rho(E_{0,k}).$$
Hence, the equiboundeness is proved and then  $ \{E_k\}_{k\in\N}$ converges up to a subsequence  to a set $E \in  \mathcal{A}_{R_1}(\omega)$  in the Hausdorff metric. Hence, by the definition of $\mathcal A_{R_1}(\omega)$, there exists an open bounded convex set  $E_0$ such that $E=E_0\setminus\overline B_{R_1}$.
In order to complete the  proof, we will prove  that 
\begin{equation}
\label{fine}
M=\lim_k\sigma_1(E_k)\leq \sigma_1(E).
\end{equation}
Let $u\in H^1_{\de B_{R_1}}(E)$ be the first positive eigenfunction associated to $\sigma_1(E)$,  such that 
\begin{equation*}
\int_{\de E_0} u^2\;d\mathcal{H}^{n-1}=1.
\end{equation*}
Hence, we have  
\begin{equation*}
\sigma_1(E)=\ds\int_{E}|\nabla u|^2\;dx.
\end{equation*}
By the extension theorem (see for instance  \cite{C,St} for Lipschitz domains), we can extend $u$ in $\mathbb{R}^n$ obtaining a function $\tilde u \in H^1_{\de B_{R_1}}\left(\R^n\right)$ such that $\tilde u=u$, a.e. in  $E$, and 
\begin{equation*}
\|\tilde u\|_{H^1_{\de B_{R_1}}\left(\R^n\right)}\le c(n) \| u\|_{H^1_{\de B_{R_1}}\left(E\right)},
\end{equation*}
for some positive constant $c=c(n)$.
 For every $k\in\mathbb{N}$ we define $u_k$ as the restriction of $\tilde{u}$ in $E_k$.
Using $u_k$ as a  test function for $\sigma_1(E_k)$, we have
\begin{equation}\label{sigma}
\sigma_1(E_k)\leq \dfrac{\ds\int_{E_k}|\nabla \tilde{u}|^2\;dx}{\ds\int_{\partial E_{0,k}}\tilde{u}^2\;d\mathcal{H}^{n-1}}.
\end{equation}
In order to get \eqref{fine},  we prove that the right-hand side in \eqref{sigma} converges to $\sigma_1(E)$. 
We  observe that 
\begin{equation}\label{gradient}
\ds\int_{E_k}|\nabla \tilde{u}|^2\;dx-\ds\int_{E}|\nabla \tilde{u}|^2\;dx=\ds\int_{\mathbb{R}^n}\left(\chi_{E_k}\;-\chi_{E}\right)|\nabla \tilde{u}|^2\;dx\to 0,
\end{equation}
since $E_k\to E$ in the Hausdorff metric and by the dominated convergence theorem.

In order to conclude the proof we have to prove 
\begin{equation} \label{denom}
	\int_{\partial E_{0,k}} \tilde{u}^2 \,d\mathcal{H}^{n-1}\to \int_{\partial E_0} u^2\,d\mathcal{H}^{n-1} = 1. 
\end{equation}
The equiboundedness of the sequence $\{E_{0,k}\}_{k\in \mathbb{N}}$ guarantees the existence of a ball $B_R$ centered at the origin with radius $R>0$ such that $E_{0,k}\subset B_R$, for every $k\in\mathbb{N}$. Extending $\tilde u$ to zero in $B_{R_1}$ and by using an approximation argument, we can suppose that $\tilde u \in C^{\infty}(B_R)$. Then \eqref{denom} follows by Theorem \ref{fr}.
Finally,  passing to the limit in \eqref{sigma}, by \eqref{gradient} 
and \eqref{denom}, we get \eqref{fine}, that is
\[
M\leq \sigma_1(E)
\]
and, consequently, we can conclude that 
\[
\sigma_1(E)=M,
\]
obtaining the desired claim.
	\end{proof}

\begin{rem} We observe that the above existence result holds even when we consider $\Omega_K=\Omega_0 \setminus K$, where  $K$ is a convex body strictly contained in $\Omega_0$. Indeed, by using the upper bound \eqref{UCH}, the proof can be done following  line by line  the one just discussed in the case of a spherical hole.\end{rem}
\section{Proof of the main result}
In this section we give the proof of the main result. The idea is to take as  test function in the quotient \eqref{eig} the eigenfunction of the spherical shell with the same measure as $\Omega$.
Before giving the proof,  we need a preliminary result.
\begin{lem} \label{convexity}Let $R_1>0$ and let $f$ be the function defined in $]0, +\infty [$ as
\begin{equation*}
    f(t)= 
    \begin{cases}
    \log^2 \left(\frac{\sqrt{t}}{ R_1}\right) \sqrt{t} &  n=2 \\
    \big( \frac{1}{R_1^{n-2}}-\frac{1}{t^{\frac{n-2}{n}}}\big)^2 \; t^{\frac{n-1}{n}} & n\ge 3.
    \end{cases}
\end{equation*}
Then, $f$ is convex for every $\alpha_-(n) R_1^n \le t \le  \alpha_+(n) R_1^n$, where
\begin{equation*}
    \alpha_\pm(n) = 
    \begin{cases}
    e^{\pm 2 \sqrt{2}}  & n=2 \\
    \bigg[\frac{(n-1)\pm (n-2)\sqrt{2(n-1)}}{n-1}\bigg]^{\frac{n}{n-2}} & n \ge 3.
    \end{cases}
\end{equation*}
\end{lem}
\begin{proof} Let us begin with the bidimensional case. After an easy computation one can see that
\begin{equation*}
    f''(t) = \frac{2-\log^2 (\sqrt{t}/ R_1) }{4t\sqrt{t}},
\end{equation*}
which gives immediately the conclusion.

Now let us consider $n\ge 3$. After some computations the second derivative of the function is the following
\begin{equation*}
    f''(t) = t^{\frac{3}{n}-3} \bigg[ \frac{R_1^{4-2n}}{n}\bigg( \frac{1}{n}-1\bigg)t^{2-\frac{4}{n}} + \frac{2R_1^{2-n}}{n} \bigg(1-\frac{1}{n}\bigg) t^{1-\frac{2}{n}} + \bigg( \frac{3}{n}-2 \bigg) \bigg( \frac{3}{n}-1 \bigg) \bigg].
\end{equation*}
If we call $y = t^{1-\frac{2}{n}}$, the previous function is non-negative if and only if 
\begin{equation*}
g(y) = \frac{R_1^{4-2n}}{n}\bigg( \frac{1}{n}-1\bigg)y^2 + \frac{2R_1^{2-n}}{n} \bigg(1-\frac{1}{n}\bigg)y + \bigg( \frac{3}{n}-2 \bigg) \bigg( \frac{3}{n}-1 \bigg) \ge 0.
\end{equation*}
It is not difficult to check that the zeros of $g(y)$ are
\begin{equation*}
    y_{\pm} = R_1^{n-2}\, \frac{n-1 \pm (n-2)\sqrt{2(n-1)}}{n-1}.
\end{equation*}
Being  $y_{-}= 0$ for $n= 3$ and $y_{-}<0$ for every $n\ge 4$ it must be $y_{-} \le y \le y_+$, which concludes the proof.
\end{proof}

Now we can prove the main result.
\begin{proof}[Proof of the Theorem \ref{main}]
Let us consider the fundamental solution $w$, given in \eqref{radial}, as a test function in \eqref{eig}. Then, 
\begin{equation*}
 \sigma_1(\Omega) \le \displaystyle \frac{\ds\int_{\Omega} |\nabla w |^2 \,dx}{\ds\int_{\partial\Omega_0} w^2  \,d \mathcal H^{n-1}}.   
\end{equation*}
In order to prove the result we will show that
\begin{equation}
\label{claim}
\frac{\ds \int_{\Omega} |\nabla w |^2 \,dx}{\ds \int_{\partial\Omega_0} w^2  \,d \mathcal H^{n-1}}\le\frac{\ds \int_{A_{R_1,R_2}} |\nabla w |^2 \,dx}{\ds \int_{\partial B_{R_2}} w^2  \,d \mathcal H^{n-1}}= \sigma_1(A_{R_1,R_2}).
\end{equation}

Since $|\nabla w|^2$ is a non-negative radially symmetric decreasing function for  any $n\ge 2$, it coincides with its Schwarz symmetrization. Hence by  the Hardy-Littlewood inequality \cite[Th. 1.2.2]{K}, 
we have
\begin{equation}
\label{hl}
\begin{split}
\int_{\Omega} |\nabla w|^2\ dx & =\int_{\Omega_0} |\nabla w|^2\ dx-\int_{B_{R_1}} |\nabla w|^2\ dx\\ 
& \leq \int_{B_{R_2}} |\nabla w|^2\ dx-\int_{B_{R_1}} |\nabla w|^2\ dx=\int_{A_{R_1,R_2}} |\nabla w|^2\ dx.
\end{split}
\end{equation}

Hence, it remains to prove the following inequality
\begin{equation}
\label{key}
\int_{\partial\Omega_0} w^2  \,d \mathcal H^{n-1}\ge \int_{\partial B_{R_2}} w^2  \,d \mathcal H^{n-1}.
\end{equation}
Let $\rho_0$ be the radial function of $\Omega_0$ defined in \eqref{ro}. By \eqref{rad}, $\partial \Omega_0$ can be represented as follows 
\[
\partial \Omega_0 = \{ x\,\rho_0(x),  x \in \mathbb S^{n-1}\},
\]
with $R_1<\rho_0(\theta)\le\tilde{R}$ and $\bar R$ defined in \eqref{bound}. \\
Firstly, let us consider the case $n=2$. If we denote  by $z(\theta)=R^2(\theta)=\rho_0^2(x(\theta))$, being $V(\Omega_0)=V(B_{R_2})$, it holds
\begin{equation}
\label{vol}
R_2=\displaystyle\sqrt{\frac{1}{2\pi}\int_0^{2\pi}z(\theta)d\theta}.
\end{equation}
Moreover,  we get
\begin{equation}
\begin{split}
\int_{\partial\Omega_0}w^2 \,ds=&\int_{\partial\Omega_0} \left(\log(|x|)-\log R_1\right)^2\ ds=\int_0^{2\pi}\log^2\left(\frac{R(\theta)}{R_1}\right)R(\theta)\sqrt{1+\left(\frac{R'(\theta)}{R(\theta)}\right)^2}\ d\theta\\
&\geq \int_0^{2\pi}\log^2\left(\frac{R(\theta)}{R_1}\right)R(\theta)\ d\theta=\int_0^{2\pi}\log^2\left(\frac{\sqrt{z(\theta)}}{R_1}\right)\sqrt{z(\theta)}\ d\theta 
\\\label{last}&\ge 2\pi\log^2\left(\frac{1}{R_1}\sqrt{\frac{\int_0^{2\pi}z(\theta)d \theta}{2\pi}}\right)\sqrt{\frac{\int_0^{2\pi}z(\theta)d\theta}{2\pi}}=\\ &= 2\pi R_2\log^2\left(\frac{R_2}{R_1}\right)=\int_{\partial B_{R_2}}w^2 \, ds,
\end{split}
\end{equation}
where, since $\rho_0(x)\le \bar R$, the last inequality follows  by Lemma \ref{convexity} and by Jensen's inequality.
This conclude the proof of  \eqref{key} in the bidimensional case. 

Now, let us consider the case  $n\ge 3$ and we proceed in a similar way.

Moreover since
\begin{equation*}
V(\Omega_0)=\frac 1 n \int_{\mathbb S^{n-1}} \rho_0^n(x) \,d\mathcal H^{n-1}
\end{equation*}
and being $V(\Omega_0)=V(B_{R_2})$, it holds
\begin{equation}
\label{voln}
R_2=\displaystyle\left(\frac{1}{n\omega_n}\int_{\mathbb S^{n-1}}z(x)\, d \mathcal H^{n-1}\right)^{\frac 1 n},
\end{equation}
where  $z(x)=\rho_0^n(x)$.
Then,  we have
\begin{equation*}
    \begin{split}
       \int_{\partial\Omega_0}w^2 \,d\mathcal H^{n-1}=& \int_{\partial \Omega_0} \bigg( \frac{1}{R_1^{n-2}} - \frac{1}{|x|^{n-2}}\bigg)^2\,d\mathcal{H}^{n-1} \\     
        &=\int_{\mathbb S^{n-1}} \bigg( \frac{1}{R_1^{n-2}} - \frac{1}{(\rho_0(x))^{n-2}}\bigg)^2  (\rho_0(x))^{n-1}\sqrt{1+\left(\frac{\nabla_\tau \rho_0(x)}{\rho_0(x)}\right)^2}\,d\mathcal H^{n-1} \\
        &\ge\int_{\mathbb S^{n-1}} \bigg( \frac{1}{R_1^{n-2}} - \frac{1}{(z(x))^{\frac{n-2}{n}}}\bigg)^2  (z(x))^{\frac{n-1}{n}}\,d\mathcal H^{n-1} \\
        & \ge n\omega_n \bigg[ \frac{1}{R_1^{n-2}} -  \frac{n\omega_n} {\left(\int_{\mathbb S^{n-1}} z(x)\,d\mathcal H^{n-1}\right)^{\frac{n-2}{n}} } \bigg]^2 \bigg( \frac{\int_{\mathbb S^{n-1}} z(x)\,d\mathcal H^{n-1}}{n \omega_n} \bigg)^{\frac{n-1}{n}}\\
	&= n\omega_n \bigg( \frac{1}{R_1^{n-2}} - \frac{1}{R_2^{n-2}} \bigg)^2 R_2^{n-1} = \int_{\partial B_{R_2}} w^2\,d\mathcal{H}^{n-1}.
    \end{split}
\end{equation*}
where last inequality follows by Lemma \ref{convexity} and by Jensen's inequality,  being $\rho_0(x)\le \bar R$. This gives \eqref{key} for $n \ge 3$ and concludes the proof. \end{proof}

\section{Some remarks about the perimeter constraint }
The estimate \eqref{Upperp} states that the first Steklov-Dirichlet eigenvalue is bounded from above also when we keep the outer perimeter and the radius of the inner ball fixed. So,  it is natural to investigate if there exists a set which maximizes $\sigma_1(\Omega)$ in the following class  
\begin{equation*}
\mathcal{B}_{R_1}(\kappa) := \big\{D = K \setminus \overline{B}_{R_1}\ , \ K \subset \mathbb{R}^n, \,\,\text{open, convex} : \ B_{R_1}\Subset K,\  P(K)= \kappa \big \},
\end{equation*}where $R_1>0$ and $\kappa>n \omega_n R_1^{n-1}.$
Arguing as Theorem \ref{esi}, we obtain the following existence result under a perimeter constraint.
\begin{thm}
	Let $\kappa>n\omega_nR_1^{n-1}$ be fixed. There exists a set $\Omega\in\mathcal{B}_{R_1}(\kappa) $ such that 
	\begin{equation*}
	\sup_{D\in\mathcal{B}_{R_1}(\kappa) }  \sigma_1(D)=\sigma_1(\Omega).
	\end{equation*}
\end{thm}

\begin{rem}
We stress that inequality \eqref{hl} continues to hold true  even if we fix the perimeter of $\Omega_0$. Indeed  the isoperimetric inequality ensures that the ball $B_{R_2}$ centered at the origin and having the same measure than $\Omega_0$ is contained in the ball centered at the origin and having the same perimeter than $\Omega_0$. 

On the other hand  we cannot prove, instead, the  inequality \eqref{key} under the perimeter constraint in order to obtain that the spherical shell is still a maximum for $\sigma_1(\Omega)$.
Indeed, if we proceed as in the  proof of Theorem \ref{main}, for instance in the planar case, equation \eqref{vol} has to be replaced by the following inequality:
\begin{equation}
\label{per}
2\pi R_2 = P(B_{R_2})=P(\Omega_0) = \int_0^{2\pi} R(\theta) \sqrt{1 + \left(\frac{R'(\theta)}{R(\theta)}\right)^2}\, d\theta \ge \int_0^{2\pi} R(\theta)\,d\theta,
\end{equation}
where $R(\theta)=\rho_0(x(\theta))$
Then, in the last step of \eqref{last}, after using  Jensen's inequality, we do not obtain the  first Steklov-Dirichlet eigenvalue of the spherical shell, since \eqref{per} is not an equality.

In support of this fact, we   give the following numerical counterexample obtained by using \texttrademark{WxMaxima}. 
We consider $R_1=10^{-5}$ and $\Omega_0$ an ellipse with the same perimeter as $A_{R_1,1}$. Let $a$ and $b$ the semi-axes of the ellipse. In order to compute the integral over the ellipse, we used the formula $P(\Omega_0)=2\pi\sqrt\frac{{a^2+b^2}}2 $, which is an approximation by excess for the perimeter of the ellipse. Here we have chosen $b=1.1$. We obtain 
\[
D(A_{R_1,1})\approx 832,820208 > 828,919156 \approx D(\Omega_0),
\]
where $D(\Omega_0)=\int_{\partial\Omega_0}w^2 ds$ and $w$ is the fundamental solution defined in \eqref{radial}.

This means that we cannot study separately the numerator and denominator terms to obtain inequality \eqref{iso_into} under perimeter constraint. 
\end{rem}



\section*{Acknowledgements}
This work has been partially supported by the MiUR-PRIN 2017 grant \lq\lq Qualitative and quantitative aspects of nonlinear PDEs\rq\rq, by GNAMPA of INdAM,  by the MiUR-Dipartimenti di Eccellenza 2018-2022 grant \lq\lq Sistemi  distribuiti  intelligenti\rq\rq of  Dipartimento  di  Ingegneria  Elettrica  e dell'Informazione \lq\lq M. Scarano\rq\rq.

Moreover, we would like to thank the reviewer for his/her suggestions to improve this
paper.

\small{

}

\end{document}